\documentclass[11pt]{article}

\usepackage{graphicx}
\usepackage{latexsym}
\usepackage{amsmath,amsthm,amssymb,amscd}
\usepackage{epsf,amsmath}
\newtheorem{lemma}{Lemma}

\newtheorem{theorem}{Theorem}

\numberwithin{table}{section}

\def \sign{{\rm sign}}

\def\var{{\rm Var}}

\newcommand{\vu}{\mbox{\boldmath $u$}}

\newcommand{\vx}{\mbox{\boldmath $x$}}
\newcommand{\vs}{\mbox{\boldmath $s$}}

\newcommand{\vbeta}{\mbox{\boldmath $\beta$}}

\newcommand{\vmu}{\mbox{\boldmath $\mu$}}

\newcommand{\vzero}{\mbox{\boldmath $0$}}

\begin{document}

\begin{center}
{\Large\bf  
Selection Consistency of EBIC for GLIM with Non-canonical Links and Diverging Number of Parameters\\}

\vspace{0.15in}

{\sc By}  SHAN LUO${}^{1}$ {\sc and}  ZEHUA CHEN${}^{2}$\\
\vspace{0.15in}

${}^{1,2}$Department of Statistics and Applied Probability \\
National University of Singapore \\
${}$ \\
Email: ${}^{1}$luoshan08@nus.edu.sg,   ${}^{2}$stachenz@nus.edu.sg. 
\vspace{0.5in}

\end{center}

\begin{quotation}
\noindent {\it Abstract:}
In this article, we investigate the properties of the EBIC in variable selection for generalized linear models with non-canonical links and diverging number of parameters in ultra-high dimensional feature space. The selection consistency of the EBIC in this situation is established under moderate conditions.  The finite sample performance of the EBIC coupled with a forward selection procedure is demonstrated through simulation studies and a real data analysis. 

\vspace{9pt}
\par
\end{quotation}

\section{Introduction}
\label{secintro}

Variable selection is a primary concern in many important contemporary scientific fields such as signal processing,  medical research and genetic studies etc..  In these fields, usually, a relatively small set of relevant variables need to be selected from a huge collection of available variables.
For example,   in genetic genome-wide association studies (GWAS),  to identify loci or genes that affect a quantitative trait  or a disease status, thousands of thousands, even minions,  of single nucleotide polymorphisms (SNP) are under consideration.   The number of variables is much larger than the sample size in such studies. This phenomenon is referred to as small-$n$-large-$p$.  Variable selection in small-$n$-large-$p$ problems poses a great challenge.  

A major approach for variable selection is model based; that is,  a model is formulated to describe the relationship between a response variable (e.g., the measurement of a quantitative trait) and a set of predictor variables or covariates (e.g., the genotypes of SNPs), and the covariates are selected by a certain variable selection criterion.  A variable selection criterion is  crucial in model based variable selection.   Traditional variable selection criteria such as Akaike's Information Criterion (AIC) (Akaike, 1973),  Bayes Information Criterion (BIC) (Schwarz (1978)) and Cross Validation (CV) (Stone (1974)) are no longer appropriate for variable selection in small-$n$-large-$p$ problems. These traditional criteria tend to select too many irrelevant covariates because they are generally not selection consistent.   Recently, some BIC-type criteria have been proposed for small-$n$-large-$p$ problems.   Bogdan et al. (2004) considered a criterion called modified BIC (mBIC) for QTL mapping models.  Wang et al. (2009) studied another modified BIC for models with diverging number of parameters.  Chen and Chen (2008) extended the original BIC to a family called  extended BIC (EBIC) governed by a parameter $\gamma$.  

The criterion considered by Wang et al. (2009) modifies the original BIC by multiplying the second term of BIC with a diverging parameter and is somehow ad hoc. To achieve selection consistency, it  requires $p/n^{\xi} <1$ for some $0<\xi<1$, and hence is not applicable when $p>n$.  The mBIC and EBIC  considered by Bogdan et al. (2004) and Chen and Chen (2008) respectively are developed from a Bayesian framework.  For the mBIC, a binomial prior on the number of covariates is imposed on each model. For EBIC, the prior on a model is proportional to a power of the size of model class which the model belongs.  Asymptotically, mBIC is a special case of EBIC corresponding to  $\gamma = 1$.  The selection consistency of EBIC for linear models with fixed number of parameters is established in Chen and Chen (2008). The result is then extended to generalized linear models (GLIM) with canonical links in Chen and Chen (2012).  The EBIC has been used for choosing tuning parameters in penalized likelihood approaches, see Huang  et al. (2010), for feature selection procedures, see Wang (2009) and Luo and Chen (2011), and for QTL mapping and disease gene mapping studies, see Li and Chen (2009) and Zhao and Chen (2012).  

 In GLIMs,  canonical links do not always provide the best fit.  Generally,  there is no reason apriori why a canonical link  should be used, and in many cases a non-canonical link is more preferable,  see McCullagh and Nelder (1989) and  Czado and Munk (2000).  
In many conventional scientific fields such as those  mentioned at the beginning of this article, it becomes a norm that the number of covariates under consideration is so large that it can be considered as having an exponential order of the sample size.  This is referred to as the case of ultra-high dimensional feature space.  In problems such as QTL and disease gene mapping,  a quantitative trait or disease status is usually affected by many loci. Except a few so-called major genes,  most of the loci have only a small effect which cannot be detected when the sample size is small. As the sample size increases, so does the number of detectable such effects.  This phenomenon is mathematically well modeled by diverging number of parameters, i.e., the number of truly relevant covariates diverges as the sample size increases.  Therefore the GLIMs with non-canonical links and diverging number of parameters in the case of ultra-high dimensional feature space become appealing.   In this article, we investigate the properties of EBIC for such models and establish its selection consistency. 
 The selection consistency of EBIC for GLIMs with canonical links does not trivially pass to the case of non-canonical links.
The selection consistency in the case of non-canonical links is established under more general conditions than those in Chen and Chen (2012). The conditions, though general, are naturally satisfied by many popular examples as given  in Wedderburn (1976).   We also present  a forward selection procedure with  EBIC for the GLIMs.   This procedure is applied in simulation studies and a real data analysis to evaluate its validity.  

The remainder of this article is organized as follows. In section 2, the main results are presented and discussed.  
In section 3,  simulation studies are reported and analyzed.  In section 4,  the forward selection procedure with EBIC is applied to analyze
a well known Leukemia data set published in Golub et al (1999).   All the technical proofs are provided in the Appendix.  

\section{Selection Consistency of EBIC for GLIM with non-canonical links}
\label{secmain}

Let $(y_i, \vx_i),  i=1, \dots, n, $ be the observations where $y_i$ is a response variable and $\vx_i = (x_{i1}, \dots, x_{ip_n})^\tau $ is a $p_n$-vector of  covariates. We consider the generalized linear model (GLIM) below:
\[ y_i \sim  f(y_i; \theta_i) = \exp\{ \theta_i y_i - b(\theta_i) \} \ \mbox{w.r.t.} \ \nu,  \ \ i =1, \dots, n, 
\] 
where $\nu$ is a $\sigma$-finite measure.  From the properties of exponential family, we have
\[ \mu (\theta_i) = E(y_i) =  b^{'}(\theta_i),  \   \sigma^2 (\theta_i)  = \var(y_i) =   b^{''}(\theta_i), \]
where $b^{'}$ and $b^{''}$ are the first and the second derivatives of $b$ respectively. 
The $\theta_i$ is related to $\vx_i$ through the relationship:
\[g(\mu (\theta_i))  = \eta_i =  \vx_i^{\tau}\vbeta,\]
where $g$ is a monotone function called link function and $\vbeta$ is $p_n$-dimensional parameter vector. If $g(\mu(\theta_i)) = \theta_i$, i.e., $ g = \mu^{-1}$, the link is called the canonical link.  In this article, we consider general link functions including the canonical link. Because of the one-to-one correspondence  between $\theta_i$ and $\eta_i$, there is a function $h$ such that $\theta_i = h(\eta_i) = h(\vx_i^{\tau}\vbeta)$.  If $g$ is twice differentiable, so is $h$.  Thus the probability density function of $y_i$ can be expressed as 
\[  f(y_i; h(\vx_i^\tau \vbeta)) = \exp\{  y_i h(\vx_i^\tau \vbeta)- b(h(\vx_i^\tau \vbeta)) \} .\]

In the above GLIM, we assume that $p_n = O(\exp\{n^\kappa\})$, $0 < \kappa < 1$,  and that there are only a relatively small number of components of $\vbeta$ are nonzero.  Throughout the article, the following notation and convention are used.  Denote by $s$ any subset of the index set ${\cal S} = \{1, 2, \dots, p_n\}$ and $|s|$ its cardinality. For convenience, $s$ is used exchangeably to denote both an index set and the set of covariates with indices in the index set, and is also referred to as a model, i.e., the GLIM consisting of the covariates in $s$.  Let $s_{0n} = \{ j: \beta_j \neq 0, j=1, \dots, p_n\}$ and $p_{0n} = |s_{0n}|$.  The covariates belonging to $s_{0n}$ are called relevant features and the others irrelevant features.  $s_{0n}$ is also referred to as the true model.  Let $X = (\vx_1^\tau, \dots, \vx_n^\tau)^\tau$. Denote by $X(s)$ the sub matrix formed by the columns of $X$ whose indices falling into $s$. Let $\vx_{i}^s$ be the vector consisting of the components of $\vx_i$ whose indices belonging to $s$, and let $\vbeta^s$ be the corresponding sub vector  of $\vbeta$.  Let $S_j$ denote the set of ${p_n \choose j}$ combinations of $j$ indices from ${\cal S}$.  Denote $\tau(S_j) = {p_n \choose j}$.

The EBIC of a model $s$, as  defined in  Chen and Chen (2008), is 
\[ \mbox{EBIC}_{\gamma}(s)=-2\ln L_n\left(\hat{\vbeta^s}\right)+|s| \ln n+2\gamma\ln \tau(S_{|s|}), \gamma\geq 0,
\]
where $L_n(\hat{\vbeta^s})$ is the maximum likelihood of model $s$ and $\hat{\vbeta^s}$ is the maximum likelihood estimate (MLE) of $\vbeta^s$.

Denote by $l_n(\vbeta^s), \vs_n(\vbeta^s)$ and $H_n(\vbeta^s)$ the log likelihood function, the score vector and the Hessian matrix of the model $s$ respectively. Suppose the link function $g$ is twice differentiable, we have
\[ \begin{split}
l_n (\vbeta^s)=& \sum_{i=1}^n [y_i h(\vx_i^{s\tau}\vbeta^s )-b (h (\vx_i^{s\tau}\vbeta^s ))] \\
\vs_n(\vbeta^s)=&\frac{\partial l_n(\vbeta^s)}{\partial \vbeta^s}=\sum_{i=1}^n[y_i-b^{'}(h(\vx_i^{s\tau}\vbeta^s))] h^{'}(\vx_i^{s\tau}\vbeta^s )\vx _i^s \\
H_n(\vbeta^s)=& -\dfrac{\partial^2l_n (\vbeta^s)}{\partial \vbeta^s \partial \vbeta^{s \tau}} \\
                      = & \sum_{i=1}^n \{ b^{''}(h(\vx_i^{s\tau}\vbeta^s) ) [h^{'}(\vx_i^{s\tau}\vbeta^s)]^2
                         -  [y_i-b^{'}(h(\vx_i^{s\tau}\vbeta^s))] h^{''}(\vx_i^{s\tau}\vbeta^s ) \} \vx_i^s\vx_i^{s\tau}\\
                      = &  H_{n1}(\vbeta^s) - H_{n0}(\vbeta^s),    \ \ \mbox{say}.
\end{split}
\]
When $s_{0n} \subset s$, we simply denote $\mu_i = b^{'}(h(\vx_i^{s\tau}\vbeta^s))$ and $\sigma_i^2 =b^{''}(h(\vx_i^{s\tau}\vbeta^s))$.  The major difference between the case of canonical links and the case of non-canonical links is as follows.  If $g$ is the canonical link, $h^{'} \equiv 1$ and $h^{''} \equiv 0$,  hence $H_{n0} \equiv 0$ and $H_n(\vbeta^s)$ is positive definite when $X(s)$ is of full column rank.  Therefore, $l_n(\vbeta^s)$ is a strictly concave function of $\vbeta^s$.  But, if $g$ is a non-canonical link, $H_n(\vbeta^s)$ is not necessarily positive definite. As a consequence,  $l_n(\vbeta^s)$ is not necessarily concave, and the maximum likelihood estimate of $\vbeta^s$ does not necessarily exist. We will show that  $H_{n0}(\vbeta^s) $ is asymptotically negligible (Lemma \ref{lemma1}) for $\vbeta^s$ in a neighborhood of the true parameter value of the GLIM. Thus  $H_n(\vbeta^s)$ is asymptotically locally positive definite.  To guarantee the existence of the MLE of $\vbeta^s$ for finite samples, we assume that the link function $g$ is chosen such that $l_n(\vbeta^s)$ has a unique maximum.  We now state the conditions required for the selection consistency of the EBIC.

\begin{description}
  \item[C1]   $\ln (p_n)=O(n^{\kappa}),  p_{0n}=O(n^{b})$ where $b\geq 0, \kappa >0$ and $b+\kappa < 1/3$;
\item[C2] $\min_{j\in s_{0n}}|\beta_{j}|\geq Cn^{-1/4}$ for some constant $C>0$;
\item[C3] For any $s$, the interior of  
${\cal B}(s)=\{\vbeta: \int \exp( h(\vx_i^{s\tau} \vbeta ) y) d \nu < \infty, i=1,2,\dots, n \}$ is not
empty. Let $\vbeta_0$ denote the true parameter of the GLIM. If  $|s| \leq k p_{0n}$, where $k>1$,  then $\vbeta_0^{s}$ is in the interior of ${\cal B}(s)$.
\item[C4]There exist positive $c_1$ and $c_2$ such that for all sufficiently
large $n$,
\[ c_1n\leq \lambda_{\min}(H_{n1}(\vbeta_0^{s\cup s_{0n}})) \leq \lambda_{\max}(H_{n1}(\vbeta_0^{s\cup s_{0n}})) \leq c_2n \] 
for all $s$ with $|s| \leq k p_{0n}$,  where
$\lambda_{\min}$ and $\lambda_{\max}$ denote respectively the smallest and
largest eigenvalues;
\item[C5]
For any given $\xi>0,$ there exists a $\delta>0$ such that when $n$
is sufficiently large,
\[(1-\xi)H_{nj}(\vbeta_0^{s\cup s_{0n}} ) \leq H_{nj}(\vbeta^{s\cup s_{0n}}) \leq(1+\xi)  H_{nj}(\vbeta_0^{s\cup s_{0n}} ), j=0,1, 
\]
whenever  $\|\vbeta^{s\cup s_{0n}} -\vbeta_0^{s \cup s_{0n}} \|_2\leq \delta$ for all $s$ with $|s| \leq k p_{0n}$; 
\item[C6] The quantities $|x_{ij}|, |h^{'}(\vx_i^{\tau}\vbeta_0)|, |h^{''}(\vx_i^{\tau}\vbeta_0)|, i =1, \dots, n; j =1, \dots, p_n$ are bounded from above, and $\sigma_i^2, i=1, \dots, n$ are bounded both from above and  below away from zero. Furthermore, 
\[
\begin{split}
\max_{1\leq j\leq p_n; 1 \leq i\leq n} \frac{  x_{ij}^2 [h^{'}(\vx_i^{\tau}\vbeta_0)]^2  }{ \sum_{i=1}^n \sigma_i^2 x_{ij}^2[h^{'}(\vx_i^{\tau}\vbeta_0)]^2 } &=o(n^{-1/3})\\
\max_{1\leq i\leq n}\frac{ [h^{''}(\vx_i^{\tau}\vbeta_0)]^2 }{ \sum_{i=1}^n\sigma_i^2[h^{''}(\vx_i^{\tau}\vbeta_0)]^2 }&=o(n^{-1/3}).
\end{split}
\]
\end{description}

Conditions C2 and C3 are the same as conditions A2 and A3 in Chen and Chen (2012). Conditions C4 - C5 reduce to conditions A4-A5 in Chen and Chen (2012) for canonical links. When A6 in Chen and Chen (2012) is satisfied,  C6 is satisfied by commonly used GLIMs such as Poisson distribution with log and power function links, Binary distribution with identity, arcsin, complementary log-log and probit links, Gamma distribution with log and inverse power function links.  These GLIMs are throughly studied in Wedderburn (1976). The verification of C6 for these GLIMs is given in a complementary document at website: \verb+ http://www.stat.nus.edu.sg/~stachenz/+.

We now state our main results as follows. 
Define ${\cal A}_0 = \{ s: s_{0n} \subset s,  s_{0n} \neq s, |s| \leq k p_{0n} \}$ and ${\cal A}_1 = \{ s: s_{0n} \not \subset s, |s| \leq k p_{0n} \}$. We have

\begin{theorem}
\label{SelectionConsistency}
Under assumptions C1-C6,  as $n\rightarrow +\infty$,
\begin{itemize}
\item[(1)] $P\left(\min_{s\in \mathcal{A}_1}\mbox{EBIC}_{\gamma}(s)\leq \mbox{EBIC}_{\gamma}(s_{0n})\right)\rightarrow 0, 
\ \ \text{for any} \ \  \gamma>0;$
\item[(2)] $P\left(\min_{s\in
\mathcal{A}_0}\mbox{EBIC}_{\gamma}(s)\leq \mbox{EBIC}_{\gamma}(s_{0n})\right)\rightarrow 0, \ \ \text{for
any}\ \  \gamma>\dfrac{1}{1-\epsilon}\left(1-\dfrac{\log n}{2\log p_n}\right),$ where $\epsilon$ is an arbitrarily small positive constant.
\end{itemize}
\end{theorem}

The following result  are needed in the proof of Theorem~\ref{SelectionConsistency}.  

\begin{lemma}
\label{lemma1}
Under conditions C1 - C6,  whenever $\|\vbeta^{s\cup s_{0n}}-\vbeta_0^{s\cup s_{0n}}\|_2 \leq \delta,$
\[
\vu^{\tau}H_n\left(\vbeta(s\cup s_{0n})\right)\vu=\vu^{\tau}H_n^E\left(\vbeta(s\cup s_{0n})\right)\vu\left(1+o_p(1)\right), 
\]
uniformly in  $ s$ with $|s| \leq k p_{0n}$. 
\end{lemma}

The above lemma imply the following result that gives the convergence rate of the $L_2$-consistency of the MLE of $\vbeta^s$ when $s_{0n} \subset s$. The result is of its own interest.
\begin{theorem}
\label{UnifRate} Under conditions C1 - C6,
as $n\rightarrow \infty$,
$\|\hat{\vbeta^s}-\vbeta_0^s\|_2=O_p(n^{-1/3}),$ uniformly for $s \in
{\cal A}_0.$
\end{theorem}

The technical details of the proof for the above results are given in the Appendix.
Theorem 1 implies that if we confine to the models with cardinality less than or equal to $kp_{0n}$ and  select the model with the smallest EBIC among all those models then, with probability converging to 1, the selected model, say, $s_n^*$, will be the same as the true model $s_{0n}$.  This property is what is called selection consistency.  The constraint that $|s| \leq k p_{0n}$ is natural since we do not need to consider any models with cardinality much larger than that of the true model in practical problems.  However, in practice, the evaluation of all models with cardinality up to $k p_{0n}$ is computationally impossible.  Like any other model selection criteria,  the EBIC is to be used in a certain model selection procedure.  In addition to the traditional forward selection procedures, a variety of procedures based on penalized likelihood approach have been developed within the last twenty years such as the LASSO (Tibishirani, 1996), SCAD (Fan and Li, 2001), Elastic Net (Zou and Hastie, 2005), and so on.  A model selection criterion can be used in these procedures to choose the penalty parameter, which corresponds to choosing a model. However, though some desirable properties such as the so-called oracle property have been established for these penalized likelihood approaches under certain conditions,  the asymptotic properties of these approaches with GLIM and ultra-high dimensional feature space have not been throughly studied yet to our knowledge. The traditional forward selection methods have been criticized for its greedy nature. But, recently, it is discovered that the greedy nature might not be bad especially when the model selection is for the selection of relevant variables rather than for a prediction model, see, e.g.,  Tropp (2004), Tropp and Gilbert (2007) and Wang (2009).  In this article, we consider the application of the EBIC with the traditional forward regression procedure for GLIM in our simulation studies and real data analysis.

\section{Simulation Study}

In our simulation studies,  we consider a GLIM with binary response and the complementary log-log link.  We take the divergent  pattern 
 $(n,p_n,p_{0n})=(n,[40 e^{n^{0.2}}],[5n^{0.1}])$ for $n=100,200,500$. The settings for the covariates, which  are adapted from Fan and Song (2010),  are described below.

\begin{description}
\item[Setting 1.]  Let $q=15$, $s_1 =\{1, \dots, q\}$, $s_2 = \left\{q+1, \dots, \left[ \frac{p_n}{3}\right] \right\}$,   $s_3 = \left\{\left[ \frac{p_n}{3}\right] +1, \dots, \left[ \frac{2p_n}{3}\right] \right\}$ and $s_4 = \left\{\left[ \frac{2p_n}{3}\right] +1, \dots, p_n\right\}$.   Let the covariate vector $\vx$ be decomposed into  $\vx = (\vx^{s_1}$, $\vx^{s_2}$, $\vx^{s_3}, \vx^{s_4})$.  Assume that $\vx^{s_1}$ follows  $N(\vzero,\Sigma_{\rho})$, where  $\Sigma_{\rho}$ has diagonal elements  1 and off-diagonal elements $\rho$,  $\vx^{s_2}$ follows $N(\vzero, I)$, the components of  $\vx^{s_3}$ are i.i.d. as a double exponential distribution with location 0 and scale 1,  the components of $\vx^{s_4}$ are i.i.d. with the normal mixture $\frac{1}{2} [ N(-1,1)+ N(1,0.5)]$.  The covariates $\vx_i^{s_k}, i=1, \dots, n$, are generated as i.i.d. copies of $\vx^{s_k}$, $k=1,2,3,4.$  Four values of $\rho$: 0, 0.3, 0.5 and 0.7, are considered. $s_{0n} = \{L\times t, t=1, \dots, p_{0n}\}$, where $L=10$.  $\beta_j = 1$,  if $j=L\times t$ with odd $t$, 1.3, if $j=L\times t$ with even $t$, 0, otherwise.

\item[Setting 2.] The same as setting 1 except $L=5$.  The essential difference between setting 1 and this setting  is that, in setting 1, all the relevant features are independent while, in this setting, three of them have pairwise correlation $\rho$. Two values of $\rho$:  0.3 and 0.5, are considered in this setting.

\item[Setting 3.]  $L=10, q=50$. In all the settings for $(n, p_n, p_{0n})$,  this $q$ is much smaller than $p_n$ and $p_n-q$ is much bigger than $L p_{0n}$.  The distribution of the covariate vector $\vx$ is specified as follows.  For $j=1, \dots, p_n-q$, the components $x_j$'s are i.i.d. standard normal variables. For $p_n-q < j \leq p_n$,  
\[ x_j =\frac{1}{5} \left[ \sum_{t=1}^{p_{0n}}(-1)^{t+1}x_{Lt}+\sqrt{25-p_{0n}} \xi_j\right],
\]
where the $\xi_j$'s  are i.i.d. standard normal variables.  $\vx_i$'s are generated as i.i.d. copies of $\vx$.  The specification for $s_{0n}$ and $\vbeta$ is the same as in setting 1.      In this setting, all the relevant features are independent, the last $q$ irrelevant features, which are highly pairwise correlated, have a weak marginal correlation with each of the relevant features but a strong overall correlation with the totality of the relevant features.
\end{description}

We apply the forward selection procedure with EBIC in the simulation studies.  In more detail, the procedure starts by fitting the GLIMs with one covariate, the covariate corresponding the model with the smallest EBIC is the first selected variable. Then GLIMs with two covariates including the first selected variable are considered, the additional covariate corresponding to the two-covariate model with the smallest EBIC is the second selected variable. The procedure continues this way and at each step, one more covariate is selected.  To reduce the amount of computation, when $p_n$ is bigger than 1000,  the sure independence screening procedure based on the maximum marginal estimator (MME) (Fan and Song (2010)) is used to reduce the dimension of the feature to 400 before the forward selection procedure is invoked. We consider four $\gamma$ values in EBIC, i.e., $\gamma_1 = 0, \gamma_2 = \frac{1}{2}( 1-\frac{\ln n}{2\ln p_n}), \gamma_3 = 1-\frac{\ln n}{4\ln p_n}$ and $\gamma_4 = 1$.   We choose these values because $\gamma_1$ corresponds to the original BIC, $\gamma_4$ corresponds to mBIC,  $\gamma_2$ is halfway between 0 and $ 1-\frac{ \ln n}{ 2 \ln p_n}$, the lower bound of the consistent range of $\gamma$,  and $\gamma_3$ is halfway between $ 1-\frac{ \ln n}{ 2 \ln p_n}$ and 1.  Thus we can evaluate the asymptotic behavior of EBIC when the $\gamma$ value is below and above the lower bound of the consistent range and also make a comparison with BIC and mBIC.  The performance of the procedure is evaluated by positive discovery rate (PDR) and false discovery rate (FDR).
The PDR and FDR are defined as follows:  
\[ \mbox{PDR}_n =\frac{ \nu(s^{*} \cap s_{0n})}{\nu(s_{0n})}, \ \ \mbox{FDR}_n =\frac{\nu(s^{*} \backslash s_{0n})}{\nu(s^*)},\]
where $s^*$ is the set of selected features.  The  selection consistency is equivalent to 
\[ \lim_{n\to \infty }  \mbox{PDR}_n  = 1  \ \ \mbox{and} \ \ \lim_{n\to \infty }  \mbox{FDR}_n  = 0, \]
in probability.    The PDR and FDR are averaged over 200 replications. The results under Settings 1-3 are reported in Tables 1- 3 respectively.

By examining Tables 1 - 3, we can find the following common trends: 1) with all the four $\gamma$ values, the PDR increases as $n$ gets larger, 2) with $\gamma_1$ and $\gamma_2$ (which are below the lower bound of the consistent range), the FDR does not show a trend to decrease while, with  $\gamma_3$ and $\gamma_4$ (which are within the consistent range), the FDR reduces rapidly towards zero, 3) though the PDRs with $\gamma_3$ and $\gamma_4$ are lower than those with $\gamma_1$ and $\gamma_2$ when sample size is small, but they become comparable as the sample size increases, and 4) the FDR with $\gamma_4$ is lower than that with $\gamma_3$ when sample size is small, however, the PDR is also lower, as sample size gets larger, both the PDR and FDR with $\gamma_3$ and those with $\gamma_4$ become comparable.  These findings demonstrate that the selection consistency of EBIC is well realized in finite sample case.

\section{Real Data Analysis}
\label{real}
In this section, we apply the forward selection procedure with EBIC to analyze a Leukemia data set. The data consists of the expression levels of 7129 genes obtained from  47 patients with acute lymphoblastic leukemia (ALL) and 25 with acute myeloid leukemia (AML). The data set is available in the R packages {\it Biobase} and {\it golubEsets}.  The initial version of this data set is described and analyzed by a method called ``neighborhood analysis" in Golub et al. (1999). The data set is later analyzed using GLIM with probit link in  Lee et al. (2003) and using GLIM with logit link in Liao and Chin (2007). 50 genes are identified as important ones affecting the types of leukemia in Golub et al. (1999), 27 genes are identified in Lee et al. (2003), and 19 genes are identified in Liao and Chin (2007).  There are only a few overlapped genes among the three identified sets.  

We analyzed the data by the forward selection procedure with four different link functions: {\it logit,  probit, cauchit} and {\it cloglog}. 
First, with each link function, the procedure was carried out until 50 genes were selected. The identified genes are reported in Table~\ref{gene1}.   These 50 genes are compared with three identified sets mentioned above.   Those which were identified in Golub et al. (1999), Lee et al. (2003) and Liao and Chin (2007) are indicated by $\star$,  $\triangle$ and  $\ast$  respectively.  There are three genes: 1834,1882, 6855, which are in all the three identified sets are selected by the forward selection procedure.  They are all among the selected genes with logit and cloglog links.  Two of them, i.e., 1834, 1882, are only among the selected genes with probit and cauchit links.  The other selected genes except two of them are in only one of the identified sets.  Note that the selected genes and their ordering are different among the four different links.  This indicates that the link function does  matter in  the selection procedure.  
Second,  we  used 8-fold cross validation to select the optimal link function among the four links. The optimal link is the logit link. Finally,  we made a final selection using EBIC with  $\gamma= 1-\frac{\ln n}{3\ln p_n})$ which is slightly bigger than the lower bound of the consistent range.  The final selected  variables together with the maximum log likelihood of the corresponding model are reported in Table~\ref{gene2}.  To compare the final selection of the logit link with the other links, the selected results with all the four links are reported. The genes selected by the logit link are 1834 and 4438.  The maximum log likelihood of the selected model with the logit link is the largest among all the four links.  Note that, the same two genes are also selected by probit link and the  gene 4438  is selected by  cloglog link.  We thus can conclude quite confidently that the two genes selected by logit link are the most important genes for studying the etiology of leukemia.

\newpage
\section{Appendix:  Technical Proofs}  

\begin{proof}[Proof of Lemma \ref{lemma1}] For any arbitrary $s\in \mathcal{A}_1,$ consider $\tilde{s}=s\cup s_{0n}.$ Let $a_{ni}$ in Lemma 1 of Chen and Chen (2012) be $h^{''}\left(\vx_i^{\tilde{s}\tau}\vbeta_0^{\tilde{s}}\right)\sign(y_i-\mu_i)/\sqrt{\sum\limits_{i=1}^n\sigma_i^2 \left(h^{''}\left(\vx_i^{\tilde{s}\tau}\vbeta_0^{\tilde{s}}\right)\right)^2},$ since $\vx_i^{\tilde{s}\tau}\vbeta_0^{\tilde{s}}=\vx_i^{\tau}\vbeta_0,$ from Condition C6, we have
\begin{equation}
\label{lemma1ine1}
P\left(\sum\limits_{i=1}^n|(y_i-\mu_i)h^{''}\left(\vx_i^{\tau}\vbeta_0\right)|\geq Cn^{2/3}\right)\leq 2\exp(-Cn^{1/3}).
\end{equation}
For any unit vector $\vu$ with length $|\tilde{s}|,$
\begin{equation}
\begin{split}
\label{lemma1ine2}
\vu^{\tau}H_{n0}(\vbeta_0^{\tilde{s}})\vu=&\sum\limits_{i=1}^n(y_i-\mu_i)h^{''}\left(\vx_i^{\tau}\vbeta_0\right)\left(\vu^{\tau}\vx_i^{\tilde{s}}\right)^2\\
\leq &\sum\limits_{i=1}^n|(y_i-\mu_i)h^{''}\left(\vx_i^{\tau}\vbeta_0\right)|\|\vx_i^{\tilde{s}}\|_2^2\\
\leq & C(k+1)p_{0n}\sum\limits_{i=1}^n|(y_i-\mu_i)h^{''}\left(\vx_i^{\tau}\vbeta_0\right)|.
\end{split}
\end{equation}
The last inequality is true because all $x_{i,j}'s$ are bounded, as assumed in Condition C6. (\ref{lemma1ine1}) and (\ref{lemma1ine2}) with Condition C5 imply that, for any $\xi>0$, there exists a $\delta>0$ such that
\begin{equation*}
\begin{split}
&P\left(\max_{s\in \mathcal{A}_1,\|\vu\|_2=1,\|\vbeta^{s\cup s_{0n}}-\vbeta_0^{s\cup s_{0n}}\|_2\leq \delta}\vu^{\tau}H_{n0}\left(\vbeta^{s\cup s_{0n}}\right)\vu\geq Cp_{0n}n^{2/3}\right)\\
\leq &P\left(\max_{s\in \mathcal{A}_1,\|\vu\|_2=1}\vu^{\tau}H_{n0}\left(\vbeta_0^{s\cup s_{0n}}\right)\vu\geq \dfrac{C}{1+\xi}p_{0n}n^{2/3}\right)\\
\leq &|\mathcal{A}_1|P(\sum\limits_{i=1}^n|(y_i-\mu_i)h^{''}\left(\vx_i^{\tau}\vbeta_0\right)|\geq \tilde{C}n^{2/3})\\
\leq &2\exp(kp_{0n}\ln p_n-\dfrac{C}{1+\xi}n^{1/3})=o(1).
\end{split}
\end{equation*}
Similar strategy applies to $s\in \mathcal{A}_0$ since
$\mathcal{A}_0=\{s\cup s_{0n}:s\in\mathcal{A}_1,0<|s|\leq (k-1)p_{0n}\}.$
That is, $\max\limits_{|s|\leq kp_{0n},\|\vu\|_2=1,\|\vbeta^{s\cup s_{0n}}-\vbeta_0^{s\cup s_{0n}}\|_2\leq \delta}\vu^{\tau}H_{n0}\left(\vbeta^{s\cup s_{0n}}\right)\vu=o_p(p_{0n}n^{2/3}).$ Combined with $p_{0n}=o(n^{1/3})$ in Condition C1 and Condition C4, we can have the desired result.

\end{proof}

\begin{proof}[Proof of Theorem \ref{SelectionConsistency}]
According to the definition of EBIC, for any model $s$,
 $\mbox{EBIC}_{\gamma}(s)\leq \mbox{EBIC}_{\gamma}(s_{0n})$ if and only if
 \begin{equation}
 \label{compare}
\ln L_n\left(\hat{\vbeta}^s\right)-\ln L_n\left(\hat{\vbeta}^{s_{0n}}\right)\geq (|s|-p_{0n})\ln n/2+\gamma\left(\ln \tau(S_{|s|})-\ln \tau(S_{p_{0n}})\right).
\end{equation}
To prove the selection consistency of EBIC, or mathematically, 
\[ 
P\left(\min_{s:|s|\leq kp_{0n},s\neq s_{0n}}\mbox{EBIC}_{\gamma}(s)\leq \mbox{EBIC}_{\gamma}(s_{0n})\right)\to 0\;\text{as}\;n\to +\infty,
\]
it suffices to show that inequality (\ref{compare}) holds with a probability converging to 0 as the sample size goes to infinity uniformly for all $s\in \mathcal{A}_0\cup \mathcal{A}_1$. This is completed by dealing with $s\in \mathcal{A}_0$ and $\mathcal{A}_1$ separately.
\vspace{0.15in}

(I)  {\it Case 1: $s\in \mathcal{A}_1$}. In this case,  inequality  (\ref{compare}) implies that
\begin{equation}
\label{compare1}
\ln L_n\left(\hat{\vbeta}^s\right)-\ln L_n\left(\hat{\vbeta}^{s_{0n}}\right)\geq -p_{0n}(\ln n/2+\gamma \ln p_n).
\end{equation}
Therefore, if we can show 
\begin{equation}
\label{compare1res}
P\left(\sup_{s\in \mathcal{A}_1}\ln L_n\left(\hat{\vbeta}^s\right)-\ln L_n\left(\hat{\vbeta}^{s_{0n}}\right)\geq -p_{0n}(\ln n/2+\gamma \ln p_n)\right)\to 0\;\text{as}\;n\to +\infty,
\end{equation}
then we will have
\[
P\left(\min_{s:s\in \mathcal{A}_1}\mbox{EBIC}_{\gamma}(s)\leq \mbox{EBIC}_{\gamma}(s_{0n})\right)\to 0\;\text{as}\;n\to +\infty.
\]
 
The key becomes to derive the order for $\sup_{s\in \mathcal{A}_1}\ln L_n\left(\hat{\vbeta}^s\right)-\ln L_n\left(\hat{\vbeta}^{s_{0n}}\right)$. 
For any $s\in \mathcal{A}_1,$ let $\tilde{s}=s\cup s_{0n}$ and $\breve{\vbeta}^{\tilde{s}}$ be
$\hat{\vbeta}^s$ augmented with zeros corresponding to the elements
in $\tilde{s}\backslash s.$ It can be seen that 
\[
\ln L_n\left(\vbeta_0^{\tilde{s}}\right)=\ln L_n\left(\vbeta_0^{s_{0n}}\right)\leq \ln L_n\left(\hat{\vbeta}^{s_{0n}}\right),\;\ln L_n\left(\hat{\vbeta}^s\right)=\ln L_n\left(\breve{\vbeta}^{\tilde{s}}\right),
\]
which leads to
\begin{equation}
\label{the1ine1}
\sup_{s\in \mathcal{A}_1}\ln L_n\left(\hat{\vbeta}^s\right)-\ln L_n\left(\hat{\vbeta}^{s_{0n}}\right)\leq \sup_{s\in \mathcal{A}_1}\ln L_n\left(\breve{\vbeta}^{\tilde{s}}\right)-\ln L_n\left(\vbeta_0^{\tilde{s}}\right).
\end{equation}
And also
$$\|\breve{\vbeta}^{\tilde{s}}-\vbeta_0^{\tilde{s}}\|_2\geq
\|\vbeta^{s_{0n} \backslash s}\|_2>\min_{j\in s_{0n}}\{|\vbeta_{j}|\}>Cn^{-1/4}.$$
The positive definiteness of $H_n(\vbeta)$, or the concavity of
 $\ln L_n(\vbeta^{\tilde{s}})$ in
$\vbeta^{\tilde{s}}$ implies
\begin{equation}
\label{inequ2}
\begin{split}
&\sup_{s\in \mathcal{A}_1}\ln L_n\left(\breve{\vbeta}^{\tilde{s}}\right)-\ln L_n\left(\vbeta_0^{\tilde{s}}\right)\\
\leq &\sup\{\ln L_n\left(\vbeta^{\tilde{s}}\right)-\ln L_n\left(\vbeta_0^{\tilde{s}}\right):\|\vbeta^{\tilde{s}}-\vbeta_0^{\tilde{s}}\|_2\geq
n^{-1/4},s\in \mathcal{A}_1\}\\
\leq &
\sup\{\ln L_n\left(\vbeta^{\tilde{s}}\right)-\ln L_n\left(\vbeta_0^{\tilde{s}}\right):\|\vbeta^{\tilde{s}}-\vbeta_0^{\tilde{s}}\|_2=
n^{-1/4}, s\in \mathcal{A}_1\}.
\end{split}
\end{equation}

To derive the order of the right hand side in the above inequality, we take the Taylor
Expansion of $\ln L_n\left(\vbeta^{\tilde{s}}\right)-\ln L_n\left(\vbeta_0^{\tilde{s}}\right)$ as follows:
\begin{equation}
\label{equation9}
\begin{split}
&\ln L_n\left(\vbeta^{\tilde{s}}\right)-\ln L_n\left(\vbeta_0^{\tilde{s}}\right)\\
=&\left(\vbeta^{\tilde{s}}-\vbeta_0^{\tilde{s}}\right)^{\tau}s_n\left(\vbeta_0^{\tilde{s}}\right)-\dfrac{1}{2}\left(\vbeta^{\tilde{s}}-\vbeta_0^{\tilde{s}}\right)^{\tau}H_{n1}\left(\vbeta^{\star\tilde{s}}\right)\left(\vbeta^{\tilde{s}}-\vbeta_0^{\tilde{s}}\right)\\
+&\dfrac{1}{2}\left(\vbeta^{\tilde{s}}-\vbeta_0^{\tilde{s}}\right)^{\tau}H_{n0}\left(\vbeta^{\star\tilde{s}}\right)\left(\vbeta^{\tilde{s}}-\vbeta_0^{\tilde{s}}\right)
\end{split}
\end{equation}
where $\vbeta^{\star\tilde{s}}$ is between $\vbeta^{\tilde{s}}$ and
$\vbeta_0^{\tilde{s}}.$ By condition C4 and C5, 
$$\left(\vbeta^{\tilde{s}}-\vbeta_0^{\tilde{s}}\right)^{\tau}H_{n1}\left(\vbeta^{\star\tilde{s}}\right)\left(\vbeta^{\tilde{s}}-\vbeta_0^{\tilde{s}}\right)\geq c_1n(1-\xi)\|\vbeta^{\tilde{s}}-\vbeta_0^{\tilde{s}}\|_2^2.$$
Lemma \ref{lemma1} implies that, for any $\vbeta^{\tilde{s}}$ such that
$\|\vbeta^s-\vbeta^{s_{0n}}\|_2=n^{-1/4},$ uniformly, there exists $0<c<c_1$ such that, with probability tending to $1$ as $n$ goes to $+\infty,$ 
\begin{equation}
\label{inequ1}
\ln L_n\left(\vbeta^{\tilde{s}}\right)-\ln L_n\left(\vbeta_0^{\tilde{s}}\right)\leq n^{-1/4}\|s_n(\vbeta_0^{\tilde{s}})\|_{+\infty}-\dfrac{c}{2}n^{1/2}(1-\xi).
\end{equation}

Now we need to find out the uniform rate for the components in the score function $s_n(\vbeta_0)$. We claim that under C1-C6,
\begin{equation}
\label{claim}
P\left(\max_{1\leq j\leq p_n}s_{n,j}^2\left(\vbeta_0\right)\geq Cn^{4/3}\right)=o(1).
\end{equation}
This claim can be seen from Lemma 1 in Chen and Chen (2012). For a fixed $j$, let $a_{ni}=x_{i,j}h^{'}\left(\vx_i^{\tau}\vbeta_0\right)/\sqrt{\sum\limits_{i=1}^n\sigma_i^2 x_{i,j}^2\left(h^{'}\left(\vx_i^{\tau}\vbeta_0\right)\right)^2}.$ From Condition C6, we have
\begin{equation*}
\begin{split}
P\left(s_{nj}\left(\vbeta_0\right)\geq Cn^{2/3}\right)
=&P\left(\sum\limits_{i=1}^na_{ni}(y_i-\vmu_i)>Cn^{2/3}/\sqrt{\sum\limits_{i=1}^n\sigma_i^2 x_{i,j}^2\left(h^{'}(\vx_i^{\tau}\vbeta_0)\right)^2}\right)\\
\leq & P\left(\sum\limits_{i=1}^na_{ni}(y_i-\vmu_i)>Cn^{1/6}\right)\leq \exp(-Cn^{1/3}).
\end{split}
\end{equation*}
The first inequality holds because of the boundedness of $x_{i,j}$ and $h^{'}.$
Consequently, when $\ln p_n=o(n^{1/3}),$ which is satisfied by C1, we have
\begin{equation*}
\sum\limits_{j=1}^{p_n}P\left(s_{nj}\left(\vbeta_0\right)
\geq Cn^{2/3}\right)=\exp(\ln p_n-Cn^{1/3})=o(1).
\end{equation*}
This completes the proof of the claim (\ref{claim}).

\quad Therefore, the right hand side of (\ref{inequ1}) is less than
$ c_1 n^{5/12}-c_2n^{1/2}$, which is less than $-Cn^{1/2}$ for some constant $C>0$. Combined with inequalities (\ref{the1ine1}) and (\ref{inequ2}) , this leads to
\[
\sup_{s\in \mathcal{A}_1}\ln L_n\left(\hat{\vbeta}^s\right)-\ln L_n\left(\hat{\vbeta}^{s_{0n}}\right)\leq -Cn^{1/2}.
\]

Since under C1, $p_{0n}\ln n=o(n^{1/3}), p_{0n}\ln p_n=o(n^{1/3}),$ 
we proved inequality (\ref{compare1res}).

\vspace{0.15in}
(II) {\it Case 2:  $s\in \mathcal{A}_0$}.  Let $m=|s|-\nu(s_{0n})$, Lemma 1 in Luo and Chen (2011) implies that, asymptotically, as $n\to +\infty,$ $\mbox{EBIC}_{\gamma}(s)\leq \mbox{EBIC}_{\gamma}(s_{0n})$
if and only if
\begin{equation}
\label{compare2}
\ln L_n\left(\hat{\vbeta}^s\right)-\ln L_n\left(\hat{\vbeta}^{s_{0n}}\right)\geq
m[0.5\ln n+\gamma\ln p_n].
\end{equation}
Therefore, it suffices to show 
\begin{equation}
\label{compare2res}
P\left(\sup_{s\in\mathcal{A}_0}\ln L_n\left(\hat{\vbeta}^s\right)-\ln L_n\left(\hat{\vbeta}^{s_{0n}}\right)\geq
m[0.5\ln n+\gamma\ln p_n]\right)\to 0\;\text{as}\;n\to \infty
\end{equation}
to obtain 
\[
P\left(\min_{s:s\in \mathcal{A}_0}\mbox{EBIC}_{\gamma}(s)\leq \mbox{EBIC}_{\gamma}(s_{0n})\right)\to 0\;\text{as}\;n\to +\infty.
\]

Note that Lemma \ref{lemma1} implies 
\begin{equation}
\begin{split}
\label{inequ3}
&\ln L_n\left(\hat{\vbeta}^s\right)-\ln L_n\left(\hat{\vbeta}^{s_{0n}}\right)\leq \ln L_n\left(\hat{\vbeta}^s\right)-\ln L_n\left(\vbeta_0^{s_{0n}}\right))\\
=& (\hat{\vbeta}^s-\vbeta_0^s)^{\tau}s_n(\vbeta_0^s)-\dfrac{1}{2}(\hat{\vbeta}^s-\vbeta_0^s)^{\tau}H_n(\tilde{\vbeta}^s)(\hat{\vbeta}^s-\vbeta_0^s)\\
\leq & (\hat{\vbeta}^s-\vbeta_0^s)^{\tau}s_n(\vbeta_0^s)-\dfrac{1-\epsilon}{2}(\hat{\vbeta}^s-\vbeta_0^s)^{\tau}H_{n1}(\tilde{\vbeta}^s)(\hat{\vbeta}^s-\vbeta_0^s),
\end{split}
\end{equation}
where $\xi$ is any arbitrarily small positive constant.
The applicability of the conclusion in C5 to simplify the right hand side of this inequality requires $\sup_{s\in \mathcal{A}_0}\|\hat{\vbeta}^s-\vbeta_0^s\|_2$ be approaching 0 as $n$ goes to infinity. We claim that under conditions C1-C6,
uniformly for $s\in\mathcal{A}_0,$ we have
\begin{equation}
\label{inequ4}
\|\hat{\vbeta}^s-\vbeta_0^s\|_2=O_p(n^{-1/3}).
\end{equation}
We will show this claim in the following. For any unit vector $u,$ let $\vbeta^s=\vbeta_0^s+n^{-1/3}\vu.$ Denote
$$\mathcal{T}=\Big\{\max_{s\in \mathcal{A}_0,\|\vu\|_2=1}\vu^{\tau}H_{n0}\left(\vbeta^s\right)\vu\leq Cp_{0n}n^{2/3}\Big\},$$
then Lemma \ref{lemma1} implies
\begin{equation}
\label{equation7}
\begin{split}
&P\left(\ln L_n\left(\vbeta^s\right)-\ln L_n\left(\vbeta_0^s\right)>0:\;\; \text{for some}\;\;u, s\in \mathcal{A}_0\right)\\
\leq &P\left(\ln L_n\left(\vbeta^s\right)-\ln L_n\left(\vbeta_0^s\right)>0:\;\; \text{for some}\;\;u, s\in \mathcal{A}_0| \mathcal{T}\right)+o(1).
\end{split}
\end{equation}
On $\mathcal{T},$ When
$n$ is large enough, for all $s\in \mathcal{A}_0,$ uniformly, we have
\begin{equation*}
\begin{split}
\ln L_n\left(\vbeta^s\right)-\ln L_n\left(\vbeta_0^s\right)=&n^{-1/3}\vu^{\tau}s_n\left(\vbeta_0^s\right)-\dfrac{1}{2}n^{1/3}\vu^{\tau}\left(n^{-1}H_{n1}\left(\tilde{\vbeta}^s\right)\right)\vu\\
&-\dfrac{1}{2}n^{-2/3}\left(\vu^{\tau}H_{n0}\left(\tilde{\vbeta}^s\right)\vu\right)\\
=&n^{-1/3}\vu^{\tau}s_n\left(\vbeta_0^s\right)-c_1(1-\xi)n^{1/3}/2+O(p_{0n})\\
\leq & n^{-1/3}\vu^{\tau}s_n\left(\vbeta_0^s\right)-cn^{1/3}
\end{split}
\end{equation*}
Hence, for some positive constant $c,$ we have
\begin{equation*}
\begin{split}
&P\left(\ln L_n\left(\vbeta^s\right)-\ln L_n\left(\vbeta_0^s\right)>0:\;\; \text{for some}\;\;\vu\right)\\
\leq  &P\left(\vu^{\tau}s_n\left(\vbeta_0^s\right)\geq cn^{2/3}:\;\; \text{for
some}\;\;\vu\right)\\
\leq &\sum\limits_{j\in s}P\left(s_{n,j}\left(\vbeta_0^s\right)\geq
cn^{2/3}\right)+\sum\limits_{j\in s}P\left(-s_{n,j}\left(\vbeta_0^s\right)\geq cn^{2/3}\right)
\end{split}
\end{equation*}
From (\ref{claim}), we know that $\sum\limits_{i\in \mathcal{A}_0}\sum\limits_{j\in
s}P\left(s_{n,j}\left(\vbeta_0^s\right)\geq cn^{2/3}\right)=o(1).$ The same for the second
term. Therefore,
\begin{equation}
\label{equation8}
P\left(\ln L_n\left(\vbeta^s\right)-\ln L_n\left(\vbeta_0^s\right)>0:\;\; \text{for some}\;\;\vu, s\in \mathcal{A}_0\right)=o(1).
\end{equation}
Because $\ln L_n\left(\vbeta^s\right)$ is a concave function for any $\vbeta^s,$ the
maximum likelihood estimator $\hat{\vbeta}^s$ exists and falls
within a $n^{-1/3}$ neighborhood of $\vbeta_0^s$ uniformly for $s\in
\mathcal{A}_0.$ Thus, we have
$P\left(\|\hat{\vbeta}^s-\vbeta_0^s\|_2=O(n^{-1/3})\right)\rightarrow 1.$

Now we can apply C5, the right hand side of (\ref{inequ3}) can be upper bounded by
\begin{equation*}
\begin{split}
 & (\hat{\vbeta}^s-\vbeta_0^s)^{\tau}s_n(\vbeta_0^s)-\dfrac{(1-\xi)(1-\epsilon)}{2}(\hat{\vbeta}^s-\vbeta_0^s)^{\tau}H_{n1}(\vbeta_0^s)(\hat{\vbeta}^s-\vbeta_0^s)\\
\leq & \dfrac{1}{2(1-\epsilon)}s_n^{\tau}(\vbeta_0^s)\{H_{n1}(\vbeta_0^s)\}^{-1} s_n(\vbeta_0^s)
\end{split}
\end{equation*}
where $\epsilon$ is an arbitrarily small positive value. Hence, the left hand side of (\ref{compare2res})
 is no more than
\begin{equation}
\label{equation10}
\begin{split}
&P\left(\dfrac{1}{2(1-\epsilon)}s_n^{\tau}(\vbeta_0^s)\{H_{n1}(\vbeta_0^s)\}^{-1} s_n(\vbeta_0^s)\geq
m[0.5\ln n+\gamma\ln p_n]\right)\\
\leq &|\mathcal{A}_0|\exp(-m(1-\epsilon)[0.5\ln n+\gamma\ln p_n])\\
\leq & \exp\left(m[(\ln (p_n-p_{0n}) -(1-\epsilon)\gamma\ln p_n-\dfrac{(1-\epsilon)}{2}\ln n]\right)
\end{split}
\end{equation}
It converges to 0 when $\gamma>\dfrac{1}{1-\epsilon}[1-\dfrac{\ln n}{2\ln p_n}]$.

\end{proof}

\begin{center}
\begin{table}[h]
\caption{ The PDR and FDR of the forward selection procedure with EBIC under simulation setting 1 (the PDR and FDR are averaged over 200 replicates, the numbers in parenthesis are standard errors) }
\label{simutable1}
\vspace{0.15in}
\begin{tabular}{cc|cc|cc|cc|cc} \hline
\multicolumn{2}{c|}{} & \multicolumn{2}{c|}{$\gamma_1$}&\multicolumn{2}{c|}{$\gamma_2$}&\multicolumn{2}{c|}{$\gamma_3$}&\multicolumn{2}{c}{$\gamma_4$} \\ \hline
$\rho$&$n$ &  $\mbox{PDR}$& $\mbox{FDR}$ & $\mbox{PDR}$& $\mbox{FDR}$ & $\mbox{PDR}$& $\mbox{FDR}$ & $\mbox{PDR}$& $\mbox{FDR}$  \\ \hline
0&100  &0.736& 0.375&0.735&0.362&0.646& 0.193&0.481& 0.074 \\
 & & (0.281)&(0.292)&(0.284) &(0.291)&(0.382)&(0.228)&(0.453)&(0.141) \\ \cline{2-10}
&200   &0.930 & 0.272&0.918& 0.223&0.879 & 0.127& 0.862& 0.078\\
 &&(0.220)&(0.252)&(0.253) &(0.215)&(0.311)&(0.147)&(0.337) & (0.108)\\  \cline{2-10}
& 500  &0.971 & 0.408&0.963 & 0.371&0.939 & 0.079&0.936 & 0.026 \\
 &&(0.135)&(0.181)&(0.163)&(0.152)&(0.231)&(0.119)&(0.238)&(0.062)\\  \hline
 0.3&100 &0.708& 0.407 & 0.708& 0.398& 0.621 & 0.196&0.471 & 0.081 \\
&&(0.298) & (0.296)& (0.298) &(0.306)& (0.384)&(0.230)&(0.442)& (0.152)\\ \cline{2-10}
 &200   &0.933 & 0.281 & 0.924 & 0.239& 0.889 & 0.143&0.855 & 0.083 \\
 &&(0.202)& (0.248)& (0.232)&(0.212)& (0.303)&(0.161)&(0.344)& (0.111) \\  \cline{2-10}
& 500  &0.969 & 0.428 &0.959 & 0.354&0.938 & 0.047&0.933 & 0.014\\
 &&(0.130)& (0.169)&(0.177) &(0.138)&(0.238) &(0.091)&(0.247)&(0.048) \\  \hline
 0.5&100 &0.712 & 0.401 & 0.711 & 0.383&0.632 & 0.201&0.451& 0.080 \\
&&(0.293)& (0.295)& (0.294)&(0.292)& (0.385)&(0.223)&(0.447) &(0.146) \\ \cline{2-10}
&200   &0.929& 0.281& 0.923 & 0.243& 0.881 & 0.128&0.858 & 0.084 \\
 &&(0.219) & (0.257)& (0.236)&(0.223)& (0.313)&(0.130)&(0.343)& (0.110)\\  \cline{2-10}
& 500  &0.967 & 0.434& 0.959 & 0.371&0.939& 0.043&0.933 & 0.006\\
 &&(0.142)&(0.166) & (0.168)&(0.147)& (0.235) &(0.085)&(0.249)&(0.031) \\  \hline
 0.7&100 &0.674 & 0.432 & 0.674& 0.414& 0.606 & 0.244&0.430& 0.092\\
&&(0.291)&(0.289) &(0.291)  &(0.287)&(0.365) &(0.241)&(0.432) &(0.144) \\ \cline{2-10}
&200   &0.931 & 0.292& 0.926 & 0.248& 0.888 & 0.148&0.874 & 0.112\\
 &&(0.196)& (0.246) & (0.218)&(0.207)&(0.295) &(0.146)&(0.314)&(0.125) \\  \cline{2-10}
& 500  &0.970 & 0.427 &0.966 & 0.365& 0.937 & 0.032&0.934 & 0.010 \\
 &&(0.134)& (0.173)& (0.150)&(0.150)& (0.234)&(0.072)&(0.240)& (0.038)\\  \hline
\end{tabular}
\end{table}
\end{center}

\begin{center}
\begin{table}[h]
\caption{The PDR and FDR of the forward selection procedure with EBIC under simulation setting 2 (the PDR and FDR are averaged over 200 replicates, the numbers in parenthesis are standard errors) }
\label{simutable2}
\vspace{0.15in}
\begin{tabular}{cc|cc|cc|cc|cc} \hline
\multicolumn{2}{c|}{} & \multicolumn{2}{c|}{$\gamma_1$}&\multicolumn{2}{c|}{$\gamma_2$}&\multicolumn{2}{c|}{$\gamma_3$}&\multicolumn{2}{c}{$\gamma_4$} \\ \hline
$\rho$&$n$ &  $\mbox{PDR}$& $\mbox{FDR}$ & $\mbox{PDR}$& $\mbox{FDR}$ & $\mbox{PDR}$& $\mbox{FDR}$ & $\mbox{PDR}$& $\mbox{FDR}$  \\ \hline
0.3 &100 &0.662 &0.424 & 0.660  &0.409&0.594& 0.233&0.492& 0.132\\
&&(0.272)&(0.287) & (0.276)&(0.286)& (0.350)&(0.237)&(0.392)&(0.195) \\ \cline{2-10}
&200  &0.931 & 0.256 & 0.926 & 0.231& 0.891& 0.111&0.881 & 0.068 \\
 &&(0.199)& (0.245)& (0.212)&(0.222)& (0.281) &(0.137)&(0.295)& (0.101)\\  \cline{2-10}
& 500  &0.973 & 0.401 & 0.967 & 0.339& 0.946 & 0.041&0.941 & 0.018\\
 &&(0.127)&(0.173) &(0.149) &(0.134)& (0.209)&(0.089)&(0.217)&(0.055) \\  \hline
0.5&100 &0.571 &0.489 & 0.570 &0.478& 0.521& 0.304&0.442& 0.189\\
&&(0.259)& (0.274)& (0.261)&(0.276)& (0.303)&(0.265)&(0.337)&(0.230) \\ \cline{2-10}
&200 &0.918& 0.272 & 0.910 & 0.239& 0.888 & 0.121&0.869 & 0.081\\
 &&(0.204) & (0.256)& (0.230)&(0.231)& (0.267)&(0.148)&(0.293)& (0.122)\\  \cline{2-10}
& 500  &0.970 & 0.402 & 0.964 & 0.351& 0.946 & 0.056&0.942 & 0.021\\
 &&(0.129)& (0.183)&(0.148) &(0.153)& (0.199)&(0.115)&(0.212)&(0.062) \\  \hline
\end{tabular}
\end{table}
\end{center}

\begin{center}
\begin{table}[h]
\caption{The PDR and FDR of the forward selection procedure with EBIC under simulation setting 3 (the PDR and FDR are averaged over 200 replicates, the numbers in parenthesis are standard errors) }
\label{simutable2}
\vspace{0.15in}
\begin{tabular}{c|cc|cc|cc|cc} \hline
 & \multicolumn{2}{c|}{$\gamma_1$}&\multicolumn{2}{c|}{$\gamma_2$}&\multicolumn{2}{c|}{$\gamma_3$}&\multicolumn{2}{c}{$\gamma_4$} \\ \hline
$n$ &  $\mbox{PDR}$& $\mbox{FDR}$ & $\mbox{PDR}$& $\mbox{FDR}$ & $\mbox{PDR}$& $\mbox{FDR}$ & $\mbox{PDR}$& $\mbox{FDR}$  \\ \hline
100 &0.586 & 0.506& 0.586 &0.484& 0.524 &0.332&0.387 &0.198\\
& (0.258)&(0.252) &(0.258)&(0.253) &(0.316)&(0.252)&(0.366)& (0.239) \\ \hline
200   &0.796 & 0.414& 0.791 & 0.386& 0.767& 0.285&0.746 & 0.221 \\
 &(0.261) & (0.282)&(0.274)& (0.273)&(0.311) &(0.247)& (0.334)&(0.228)\\  \hline
 500  &0.946 & 0.479& 0.936 & 0.416& 0.912 & 0.195&0.896 & 0.171 \\
 & (0.167)& (0.165)&(0.197)& (0.150)&(0.248)&(0.185)&(0.269)&(0.176) \\  \hline
\end{tabular}
\end{table}
\end{center}

\begin{center}
\begin{table}[h]
\caption{ Analysis of  Leukemia Data: the top 50 genes selected by the forward selection procedure with the four links:  logit (lo), probit (pr),  cauchit (ca) and cloglog (cl)}
\label{gene1}
\vspace{0.15in}
\begin{tabular}{ccccccccccc} \hline
       &\multicolumn{10}{c}{Rank and Gene ID}  \\ \hline
     & 1 & 2 & 3& 4& 5& 6& 7 & 8& 9 & 10   \\ \hline
lo & $1834^{\ast \triangle\star}$ & 4438 & 4951 & $6539^{\star}$ & 155 & 2181 & $1882^{\ast \triangle\star}$ & 6472 & 65 & 1953 \\
pr & $1834^{\ast \triangle\star}$ & 4438& 4951& 155& 5585&5466& 706&  $7119^{\star}$&3119 & 4480 \\
ca &
 $1882^{\ast \triangle\star}$ & 4951&$6281^{\star}$& 4499& 4443 &$ 6539^{\star}$& 5107& $1834^{\ast \triangle\star}$&4480& 6271 \\
cl &
$ 1834^{\ast \triangle\star}$& $6855^{\ast \triangle\star}$& 4377&5122 & 2830& 4407& 4780& 6309& $ 4973^{\star}$& 715 \\ \hline
 & 11 & 12 & 13& 14& 15 &16& 17 & 18 & 19 & 20 \\ \hline
lo & 3692& 706 & 1787 & $5191^{\star}$& 1239 &3119& 2784& 1078& 3631& 6308  \\
pr 
&$ 6201^{\triangle}$&490& 6895& $1882^{\ast \triangle\star}$& 1809& 2855& 3123& $ 4211^{\ast}$&$ 2020^{\ast\star}$& 3631 \\
ca
 & 6378 & 3631& $2111^{\star}$&$  6201^{\triangle}$& $ 6373^{\star}$& 1800& 4780& 321 & $4107^{\triangle}$& $1779^{\triangle} $\\
cl
 & 5376& 930& 1800& $1882^{\ast\triangle \star}$& 5794& 4399& $4389^{\star} $& 922&  1962& 4267 \\ \hline 
 &21 & 22& 23 & 24 & 25 & 26 & 27 & 28 & 29 & 30 \\ \hline 
 lo &$ 6373^{\star}$&$1909^{\star}$& 4153& $ 1685^{\triangle}$& $ 6855^{\ast\triangle\star}$&  7073& 5539& 2830& 4819&6347 \\
pr  & 5823& 1953 &$1745^{\triangle\star}$& 65&997&$ 1928^{\star}$&3307& 1787& 538& 5539 \\
ca & 6277& 1544&$ 5254^{\star}$& $ 1928^{\star}$&$  1745^{\triangle\star}$&  3163& 7073& 310&$ 4389^{\star}$&  5146 \\
cl  & 1926 &4229&$ 5254^{\star}$& 770& 2141& 6923& 7073& 2828&$ 4847^{\star}$& 698 \\ \hline 
 & 31 & 32 &33 &34 &35 &36 &37 &38 &39 & 40 \\ \hline 
lo &1081& 1095& 5328& 4279& 4373& 5737 &4366 & 5280 & 3307 &284 \\
 pr & 4107& 2385& 1087& $1909^{\star}$& 5376 & 5552& 6005& 1604& 3391& 5442 \\
ca &1927& 885& 3137 & 2258 & 4334& 6657& 2733& 5336 & 5972 & 6167 \\
cl &1779 &$1928^{\star}$& 4049& 876&  6857& 6347& $ 6376^{\star}$& 2361& 4664& 758  \\ \hline 
& 41 &42 &43 &44 & 45 & 46 & 47 &48& 49 & 50 \\ \hline
lo& 6676 &4291& 1945&4079& 3722& 668 &782& $ 4196^{\star}$&25 & $ 4389^{\star} $ \\ 
 pr & 6702 & 6309 & $2348^{\star}$ & 4282& 4925& 6167& 2323& 1779& 5122&$ 3847^{\star} $\\ 
ca & 4229&$ 4328^{\star}$&  715& 4149& $ 5191^{\star}$& 6283& 200&  6702& 5794& 4190  \\ 
cl & 3631& 6308& 4499& 4480& 5971& 6510& 5300& 3475& 3932& 6801 \\ \hline
\end{tabular}
\end{table}
\end{center}

\begin{center}
\begin{table}[h]
\caption{Analysis of  Leukemia Data: the final selected genes by EBIC}
\label{gene2}

\vspace{0.15in}
\begin{tabular}{c|c|c} \hline
Link Function &  Selected Genes & Maximum Likelihood  \\ \hline 
logit &$ 1834, \ \  4438$     & -2.296e-08 \\
probit & $1834, \ \ 4438$& -3.022e-08  \\
cauchit & $1882, \ \ 4951$ & -2.122e-06 \\
cloglog &  $1834, \ \  6855$& -6.908e-08 \\ \hline
\end{tabular}
\end{table}
\end{center}

\end{document}